\documentclass[reqno,11pt,psamsfonts]{amsart}

\usepackage{fullpage,dsfont,color}
\usepackage{amssymb,amscd,amsmath,amsthm,epsf}
\usepackage[all]{xy}

\usepackage{hyperref}

\numberwithin{equation}{section}

\chardef\bslash=`\\ 

\newtheorem{theorem}{Theorem}[section]
\newtheorem{corollary}[theorem]{Corollary}
\newtheorem{lemma}[theorem]{Lemma}
\newtheorem{proposition}[theorem]{Proposition}
\theoremstyle{remark}

\newtheorem{remark}[theorem]{Remark}

\theoremstyle{definition}
\newtheorem{definition}[theorem]{Definition}

\newcommand{\bp}{\begin{proof}}
\newcommand{\ep}{\end{proof}}

\newcommand{\thmref}[1]{Theorem~\ref{#1}}
\newcommand{\secref}[1]{Section~\ref{#1}}
\newcommand{\proref}[1]{Proposition~\ref{#1}}
\newcommand{\lemref}[1]{Lemma~\ref{#1}}
\newcommand{\corref}[1]{Corollary~\ref{#1}}

\newcommand{\un}{\mathds{1}}

\newcommand\aaa{\mathfrak{a}}

\newcommand{\aid}{\mathfrak{a}}
\newcommand{\bid}{\mathfrak{b}}

\newcommand{\pid}{\mathfrak{p}}

\newcommand\AAA{\mathcal{A}}

\newcommand{\C}{\mathbb C}
\newcommand{\N}{\mathbb N}
\newcommand{\Q}{\mathbb Q}
\newcommand{\R}{\mathbb R}
\newcommand{\T}{\mathbb T}
\newcommand{\Z}{\mathbb Z}

\newcommand\B{\mathcal{B}}
\newcommand{\hh}{\mathcal H}

\newcommand{\OO}{\mathcal O}
\newcommand{\oo}{\mathcal O}

\newcommand\ak{{\mathbb A}_K}
\newcommand\akf{{\mathbb A}_{K,f}}
\newcommand\aks{{\mathbb A}_K^*}
\newcommand\jkf{\akf^*}
\newcommand\ohs{{\hat{\OO}^*}}

\newcommand{\Pk}{P_K^+}
\newcommand{\pk}{P_K^+}
\newcommand{\Po}{P_{\OO}^+}
\newcommand{\po}{P_{\OO}^+}
\newcommand\gal{\mathcal G}
\newcommand\kab{K^{\mathrm{ab}}}

\newcommand\Mat{\operatorname{Mat}}

\newcommand{\inv}{^{-1}}

\newcommand{\kk}{K}

\newcommand\clkp{\mathrm{Cl}^+_K}

\newcommand{\kmo}{\kk/\OO}

\newcommand{\mko}{V_{K,f}}

\newcommand\os{\oo^*}
\newcommand\ox{\oo^\times}

\newcommand\kps{K_+^*}
\newcommand\kpsone{K_{+,1}^*}
\newcommand\op{\oo_+}
\newcommand\ops{\oo_+^*}

\newcommand\hatok{\hat{\oo}}
\newcommand\opx{\oo_+^\times}
\newcommand\opsbar{\overline{\oo_+^*}}
\newcommand{\heck}{C^*_r(\pk,\po)}

\newcommand{\lspa}{\operatorname{span}}

\newcommand{\matr}[2]{\left(\begin{matrix}1&#1 \\
0&#2\end{matrix}\right)}
\newcommand{\smatr}[2]{\bigl(\mbox{\tiny{$\begin{matrix}1&#1 \\
0&#2\end{matrix}$}}\bigr)}

\newcommand{\stabop}[1]{{\ops}_{\!#1}}

\newcommand\bpmatrix{\begin{pmatrix}}
\newcommand\epmatrix{\end{pmatrix}}

\begin{document}

\date{April 5, 2018; minor changes: December 9, 2018}

\title{Ground states of groupoid $C^*$-algebras, phase transitions and arithmetic subalgebras for Hecke algebras}

\author{Marcelo Laca}\address{Department of Mathematics and Statistics, University of
Victoria, P.O. Box 1700 STN CSC, Victoria, British Columbia, V8W 2Y2, Canada.}
\email{laca@uvic.ca}

\author{Nadia S. Larsen}\address{Department of Mathematics, University of Oslo,
P.O. Box 1053 Blindern, N-0316 Oslo, Norway.}
\email{nadiasl@math.uio.no}

\author{Sergey Neshveyev}\address{Department of Mathematics, University of Oslo,
P.O. Box 1053 Blindern, N-0316 Oslo, Norway.}
\email{sergeyn@math.uio.no}

\maketitle
\centerline{\em Dedicated to Alain Connes on the occasion of his 70th birthday}

\begin{abstract}
We consider the Hecke pair consisting of the group $\pk$ of affine transformations of a number field~$K$ that preserve the orientation in every real embedding and the subgroup $\po$ consisting of transformations with algebraic integer coefficients.
The associated Hecke algebra $\heck$ has a natural time evolution $\sigma$, and we describe the corresponding
phase transition for KMS$_\beta$-states and for ground states.
From work of Yalkinoglu and Neshveyev it is known that a Bost-Connes type
system associated to $K$ has an essentially unique arithmetic subalgebra. When we import this subalgebra through the
isomorphism of $\heck$ to a corner in the Bost-Connes system established by Laca, Neshveyev and Trifkovi\'c,
we obtain an arithmetic subalgebra of $\heck$ on which ground states exhibit the `fabulous' property with respect to an action of the Galois group $\gal(\kab/H_+(K))$, where $H_+(K)$ is the narrow Hilbert class field.

In order to characterize the ground states of the $C^*$-dynamical system $(\heck,\sigma)$, we obtain first a characterization of the ground states of a groupoid $C^*$-algebra, refining earlier work of Renault. This is independent from number theoretic considerations, and may be of interest by itself in other situations.
\end{abstract}

\section*{Introduction}
The seminal work of Bost and Connes \cite{bos-con}, that showed how a
Hecke $C^*$-algebra constructed from the inclusion of the ring of
integers in the field of rationals relates to class field theory for
the field $\Q$, inspired several generalizations to algebraic number
fields. One such generalization, for quadratic imaginary fields, by
Connes, Marcolli, and Ramachandran~\cite{CMR} eventually led to
the construction of Bost-Connes type systems for arbitrary number fields by Ha and
Paugam~\cite{ha-pa}, which exhibit the expected phase transition and
symmetries~\cite{LLN}. Later, Yalkinoglu \cite{Y} proved that these
Bost-Connes systems contained arithmetic models, which were
subsequently shown to be unique under some natural conditions by
Neshveyev, see the Appendix to \cite{Y}.

In the decades following the Bost-Connes paper there has been a
vast amount of work putting forward constructions of $C^*$-algebras
from number theory and unraveling relationships between operator
algebras and number theory, formalized in
classification of KMS-states, computation of $K$-theory, and in results
about recovery of number theoretical input from the operator algebra
side. Without attempting to be exhaustive, we recall some
developments.

One successful vein of investigation started with the work by
Cuntz~\cite{cun} on a $C^*$-algebra of the $ax+b$ semigroup over
$\N$, and continued with \cite{cun-li1, cun-li2, luck-li}.  The
paper \cite{LR} by Laca and Raeburn marked the first $C^*$-algebra
with a time evolution admitting an interesting  phase transition at
infinity for ground states, see also \cite{ln2}. Further results of
this type were achieved in \cite{CDL}.

It is natural to ask  what information passes between number theory
and operator algebras when one takes a number field as a starting
point and associates $C^*$-algebras to it. In \cite{li}, Li showed that the Toeplitz type algebra
associated to a number field~\cite{CDL} recovers the field up
to arithmetic equivalence given a specified number of roots of
unity. Motivated by a study of
Bost-Connes type systems, Cornelissen et al.~\cite{CdSLMS1} showed
that the number field can be recovered from the dual of the
abelianized Galois group and the collection of all associated
$L$-series. Recently, Takeishi used $K$-theoretical techniques inspired
by \cite{li} to prove that the $C^*$-algebra in the Bost-Connes
system associated to a number field $K$ recovers the Dedekind zeta-function of $K$, \cite{tak}, thus refining the findings of
\cite{LLN} that showed that the entire Bost-Connes system determines
the Dedekind zeta-function. Finally, the full reconstruction of
number fields up to isomorphism from the $C^*$-algebras of the
respective Bost-Connes systems has been recently achieved by Kubota
and Takeishi~\cite{KT}, see also~\cite{CdSLMS2}.

Notwithstanding these spectacular developments, one disadvantage of general
Bost-Connes systems is that, as opposed to the original Bost-Connes
system for $\Q$, they take class field theory as an input. By
contrast, systems based on Hecke algebras rely on simple arithmetic
considerations and therefore carry a lot of appeal. In this paper we
consider one such system already studied in~\cite{LNT} (see also~\cite{LvF,LvF2}), which is probably the most successful generalization of the original construction of Bost and Connes to general number fields.

Let $\kk$ be an algebraic number field with ring of integers $\oo$ and denote by $\kps$ and $\ops$ the respective multiplicative groups of totally positive elements, that is, elements that embed as positive numbers in every real embedding of $\kk$. We consider the $C^*$-dynamical system based on the Hecke $C^*$-algebra of the inclusion of affine groups
\[
\po:= \matr{\oo}{\ops} \ \subset \
 \matr{\kk}{\kps} =: \pk
\]
with respect to the natural dynamics determined by the modular function of the inclusion. The purpose of this paper is to give an explicit description of the phase transition of KMS-states and of ground states of this system and to explore the arithmetic subalgebra and `fabulous' property of ground states.
It has already been  observed in \cite{LNT} that this $C^*$-dynamical system exhibits a phase transition for finite inverse temperatures replicating that for Bost-Connes systems. Since the methods used in~\cite{LNT} depend in a crucial way on results for the Bost-Connes systems~\cite{LLN} and induction through $C^*$-correspondences~\cite{ln}, it is interesting to examine the phase transition, the symmetries and the arithmetic subalgebra explicitly at the level of Hecke systems.

The Hecke $C^*$-algebra $\heck$ has a realization as a groupoid $C^*$-algebra,
and in order to exploit this in the characterization of equilibrium states
we develop tools to characterize ground states and KMS$_\beta$-states
for dynamics determined by $1$-cocycles on certain groupoid $C^*$-algebras. These groupoids are often nonprincipal, and in  \secref{groupoids} we build upon the results of \cite{nes} to this effect. We believe our results will be of independent interest in other contexts, so we have presented them in a way that is independent from our original number theoretic motivation.
The parametrization of KMS-states from \cite{nes} is restated in \thmref{thm:KMS} and a similar characterization for ground states in terms of a \emph{boundary set} of a $1$-cocycle introduced by Renault~\cite{ren1} is obtained in \thmref{thm:ground}.
The specific results that are useful for our applications relate to groupoids obtained by reduction of a transformation group to a subset that is not invariant. In \thmref{boundaryandgroundforc}  we show that  ground states are parametrized by all states on the $C^*$-algebra of the \emph{boundary groupoid} and, under some extra assumptions, KMS$_\beta$-states correspond to the tracial states of the boundary groupoid, \thmref{KMSandgroundforY}. The boundary groupoid often has a noncommutative $C^*$-algebra, in which case there is a phase transition for ground states.

In \secref{KMShecke} we review the definition and basic properties of the Hecke $C^*$-algebra $\heck$, emphasizing its presentation in terms of distinguished double cosets, \proref{Heckepresentation}. This leads  to realizations as a semigroup crossed product and as a groupoid $C^*$-algebra, \proref{prop:iso}. The study of the boundary groupoid of the cocycle determined by the norm leads naturally to the consideration of integral ideals of minimal norm in their narrow class  (of fractional ideals modulo principal ideals generated by totally positive elements of $K$). This allows us to realize the boundary $C^*$-algebra explicitly as a subhomogeneous $C^*$-algebra in \corref{cor:boundaryalg}. We state and prove our explicit characterization of equilibrium states, \thmref{thm:gs-for-redHecke}, which is our main result in \secref{KMShecke}.

We begin \secref{arithmetic} by making explicit various symmetry actions on $\heck$. We also exhibit the isomorphism of $\heck$ to a corner in the Bost-Connes $C^*$-algebra $A_K$ from \cite{LNT} in an explicit form that  suits our present purposes. The main result is \thmref{thm:arithm} where we show that if we pull back the corresponding corner of the arithmetic subalgebra for the Bost-Connes system for $K$ obtained in \cite{Y} through this isomorphism, then we obtain an  arithmetic subalgebra of $\heck$ for which the ground states that are invariant under the dual coaction of $\kps$ have the expected `fabulous' property with respect to the natural action of $\gal(\kab/H_+(K))$, where $H_+(K)$ is the narrow Hilbert class field.

To summarize, we see now that the systems defined by the Hecke pairs $(\pk,\po)$ have almost as nice properties as the full Bost-Connes systems, including considerations with arithmetic models. An interesting new phenomenon is that the equilibrium states with good arithmetic properties are neither the KMS$_\infty$-states, nor the ground states, but a particular class of the latter ones.

\smallskip\noindent
{\bf Acknowledgement} 
This research was carried out through visits of M.L. to the Department of Mathematics at the University of Oslo and of N.S.L. to the Department of Mathematics and Statistics at the University of Victoria. Both are grateful to their respective hosts for their hospitality.

\section{Equilibrium states on groupoid algebras}\label{groupoids}

Let $G$ be a locally compact second countable (Hausdorff) groupoid. The unit space of $G$ will be denoted by $G^{(0)}$ and the range and source maps by $r: G \to G^{(0)}$ and $s: G \to G^{(0)}$, respectively. We will assume that $G$ is  \'{e}tale, that is, $r$ and $s$ are local homeomorphisms, which among other things implies that $G^{(0)}$ is open in $G$. For $x\in G^{(0)}$ put
$$
G^x=r^{-1}(x),\quad G_x=s^{-1}(x),\quad \hbox{and}\ \ G^x_x=G^x\cap G_x.
$$
Notice that  $G^x_x$ is a group for each $x\in G^{(0)}$, called the isotropy subgroup at $x$.
The space $C_c(G)$ of continuous functions on $G$ with compact support is a $*$-algebra under the convolution product
$$
(f_1*f_2)(g)=\sum_{h\in G^{r(g)}}f_1(h)f_2(h^{-1}g)
$$
and the involution $f^*(g)=\overline{f(g^{-1})}$. Its $C^*$-enveloping algebra is the full groupoid $C^*$-algebra $C^*(G)$.

Suppose $c\colon G\to\nolinebreak\R$ is a continuous $1$-cocycle on $G$; by definition, this means that $c$ is a continuous groupoid homomorphism of $G$ to $(\R,+)$. Then there is a unique continuous one-parameter group $\{\sigma^c_t: t\in \R\}$ of automorphisms of $C^*(G)$ such that
\begin{equation}\label{cocycledynamics}
\sigma^c_t(f)(g)=e^{it c(g)}f(g)\ \ \text{for}\ \ f\in C_c(G)\ \text{and}\ g\in G.
\end{equation}

Recall that a $C^*$-algebra $A$ equipped with a continuous one-parameter group of automorphisms $\{\sigma_t\}_{t\in \R}$ is referred to as a $C^*$-dynamical system $(A,\sigma)$. As these
are the models for quantum statistical mechanical time evolutions, their equilibrium states, or KMS-states, are of special interest. We refer the reader to \cite{bra-rob2} for a comprehensive treatment of the theory of KMS-states, and restrict ourselves to stating their definition in order to set the notation.

For each value of an inverse temperature parameter $\beta\in \R$, the  $\sigma$-KMS$_\beta$-states of $A$, alternatively, the KMS$_\beta$-states of $(A,\sigma)$, are the states $\phi$ of $A$ such that
\[
\phi(ab) = \phi(b \sigma_{i\beta}(a)),
\]
for all $b\in A$ and all analytic elements $a$ in $A$, which are those such that the map $\R\ni t \mapsto \sigma_t(a)$ extends to an entire function on $\C$. There are also related notions of KMS$_\infty$-states and ground states
that will be addressed shortly.

The KMS$_\beta$-states of $(C^*(G), \sigma^c)$ for finite $\beta$
are characterized, in terms of measures on the unit space
and traces on the $C^*$-algebras of the isotropy subgroups,
by the following theorem, which refines
an earlier result of Renault~\cite[Proposition~II.5.4]{ren1}.

\begin{theorem}[{\cite[Theorem~1.3]{nes}}] \label{thm:KMS}
Suppose that  $c$ is a continuous $\R$-valued $1$-cocycle on a locally compact second countable \'{e}tale groupoid $G$. Let $\sigma^c$ be the corresponding dynamics on~$C^*(G)$. Then, for every $\beta\in \R$, there exists a one-to-one correspondence between the $\sigma^c$-KMS$_\beta$-states on~$C^*(G)$ and the pairs $(\mu,\{\varphi_x\}_{x\in G^{(0)}})$ consisting of a probability measure $\mu$ on $G^{(0)}$ and a $\mu$-measurable field of states~$\varphi_x$ on $C^*(G^x_x)$ such that:
\begin{enumerate}
\item[(i)] $\mu$ is quasi-invariant with Radon--Nikodym cocycle $e^{-\beta c}$;
\item[(ii)] $\varphi_x(u_g)=\varphi_{r(h)}(u_{hgh^{-1}})$ for $\mu$-a.e.~$x$ and all $g\in G^x_x$ and $h\in G_x$; in particular, $\varphi_x$ is tracial for $\mu$-a.e.~$x$.
\end{enumerate}
Namely, the state corresponding to $(\mu,\{\varphi_x\}_x)$ is given~by
$$
\varphi(f)=\int_{G^{(0)}}\sum_{g\in G^x_x}f(g)\varphi_x(u_g)d\mu(x)\ \ \hbox{for}\ \ f\in C_c(G).
$$
\end{theorem}

\begin{remark}
KMS$_\beta$-states for $\beta\ne0$ are automatically $\sigma$-invariant. In the setting of the above theorem this implies that
\begin{equation}\label{eq:KMSinv}
\varphi_x(u_g)=0\ \ \text{for}\ \ \mu\text{-a.e.}\ \ x\in G^{(0)}\ \ \text{and all}\ \ g\in G^x_x\setminus c^{-1}(0).
\end{equation}
For $\beta=0$, $\sigma$-invariance is sometimes included in the definition of KMS$_0$-states. In this case \eqref{eq:KMSinv} should be added to conditions (i) and (ii) of the theorem, see~\cite{nes}.
\end{remark}

We aim to obtain next an analogous description of the set of ground states.
For this purpose it will be most convenient for us to use the characterization
of ground states given in \cite[Proposition~5.3.19~(4)]{bra-rob2}.
It states that given a continuous one-parameter group $\sigma$ of automorphisms of  a $C^*$-algebra~$A$, a state $\varphi$ of $A$ is a $\sigma$-ground state iff for every smooth compactly supported function~$F$ on $\R$ with support in $(-\infty,0)$ we have
$$
\varphi(\sigma_{\check F}(a)^*\sigma_{\check F}(a))=0\ \ \text{for all}\ \ a\in A,
$$
where $\check F$ is the inverse Fourier transform of $F$, so
$$
\check F(x)=\frac{1}{2\pi}\int_\R F(y)e^{-ixy}dy,
$$
and
$$
\sigma_{\check F}(a)=\int_\R \check F(t)\sigma_t(a)dt.
$$
The main advantage of this formulation is that when $A = C^*(G)$, with the dynamics $\sigma^c$
determined as above by a continuous $1$-cocycle $c$ on $G$, the operators $\sigma_{\check F}$
have a very simple form:
\begin{equation}\label{eq:sigmaF}
\sigma^c_{\check F}(f)(g)=F(c(g))f(g)\ \ \text{for}\ \ f\in C_c(G)\ \text{and}\ g\in G.
\end{equation}

Following Renault~\cite{ren1} we consider the subset
$$
Z =\{x\in G^{(0)}\mid c\le0\ \text{on}\ G^x\}=\{x\in G^{(0)}\mid c\ge0\ \text{on}\ G_x\}\subset G^{(0)},
$$
which we call the \emph{boundary set} of the cocycle $c$. (This is denoted by $\operatorname{Min}(c)$ in~\cite{ren1}, but we shall eschew this notation here.)
The boundary set can very well be empty, but when it is nonempty, we may consider the corresponding reduction of $G$:
$$
G_Z=\{g\in G\mid  r(g)\in Z,\ s(g)\in Z\},
$$
which we call the \emph{boundary groupoid} of the cocycle $c$. In general, $G_Z$ may fail to be \'{e}tale, so it may not have a good $C^*$-algebra associated with it. However, the kernels $c^{-1}(0)$ of many naturally arising cocycles are  \'{e}tale groupoids and we have the following result.

\begin{lemma}\label{kernellemma}
The set $Z$ is invariant under $c^{-1}(0)$ and the boundary groupoid~$G_Z$ coincides with the reduction of $c^{-1}(0)$ by $Z$. As a consequence, if the kernel groupoid $c^{-1}(0)$ is \'{e}tale, then so is the boundary groupoid $G_Z$.
\end{lemma}

\begin{proof}
The first part of the lemma is contained in~\cite[Proposition~I.3.16]{ren1}. For the second part we just need to observe that the property of being \'{e}tale is inherited by reduction to invariant subsets. Indeed, if $W\subset c^{-1}(0)$ is open and $r$ defines a homeomorphism of $W$ onto~$r(W)$, then $W\cap r^{-1}(Z)$ is contained in $G_Z$ and $r$ defines a homeomorphism of $W\cap r^{-1}(Z)$ onto $r(W)\cap Z$. Hence $G_Z$ is \'{e}tale provided $c^{-1}(0)$ is \'{e}tale.
\end{proof}

The following characterization of ground states is a refinement of
\cite[Proposition~II.5.4]{ren1}.

\begin{theorem}\label{thm:ground}
Suppose $c$ is a continuous $\R$-valued $1$-cocycle on a locally compact second countable \'{e}tale groupoid $G$.
Let $\sigma^c$ be the associated dynamics on $C^*(G)$ as in \eqref{cocycledynamics} and let $Z\subset G^{(0)}$ be the boundary set of $c$.  If $Z\neq \emptyset$ and the boundary groupoid $G_Z$ is \'{e}tale, then there is an affine isomorphism  of the state space of $C^*(G_Z)$ onto the $\sigma^c$-ground state space of $C^*(G)$ that maps the state $\psi$ of $C^*(G_Z)$ to the unique state $\varphi_\psi$ of $C^*(G)$
defined by
\begin{equation}\label{psimapstophi}
\varphi_\psi(f)=\psi(f|_{G_Z})\ \ \text{for}\ \ f\in C_c(G).
\end{equation}
If $Z=\emptyset$, then there are no $\sigma^c$-ground states on $C^*(G)$.
\end{theorem}

\bp
Assume $\varphi$ is a $\sigma^c$-ground state. Consider the corresponding GNS-triple $(H,\pi,\xi)$. By Renault's disintegration theorem~\cite[Theorem~II.1.21]{ren1} the representation~$\pi$~is the integrated form of a representation of $G$ on a measurable field of Hilbert spaces $H_x$, $x\in G^{(0)}$, with respect to a measure class~$[\nu]$ on~$G^{(0)}$. Identifying $H$ with~$\int^\oplus_{G^{(0)}}H_x\,d\nu(x)$, consider the vector field~$(\xi_x)_x$ defining~$\xi$. Then, for every $f\in C_c(G)$, we have
\begin{equation} \label{eq:group-rep}
(\pi(f)\xi)_x=\sum_{g\in G^x}D(g)^{-1/2}f(g)g\xi_{s(g)},
\end{equation}
where $D$ is the Radon--Nikodym cocycle defined by the quasi-invariant measure $\nu$.

We claim that $\xi_x=0$ for $\nu$-a.e.~$x\in G^{(0)}\setminus Z$. Take a point $x_0\in G^{(0)}\setminus Z$. There exists $g_0\in G_{x_0}$ such that $c(g_0)<0$. Choose an open neighborhood $W$ of $g_0$ such that $r|_W\colon W\to r(W)$ and $s|_W\colon W\to s(W)$ are homeomorphisms and $c(W)\subset(c(g_0)-\delta,c(g_0)+\delta)$ for some $\delta\in(0,-c(g_0))$. Take any smooth compactly supported function $F$ on $\R$ such that $\operatorname{supp}F\subset(-\infty,0)$ and $F\equiv1$ on $[c(g_0)-\delta,c(g_0)+\delta]$. Then, for any function $f\in C_c(G)$ with support in $W$, it follows from~\eqref{eq:group-rep} that
$$
(\pi(f)\xi)_x=\begin{cases}D(r^{-1}(x))^{-1/2}f(r^{-1}(x))r^{-1}(x)\xi_{s(r^{-1}(x))},&\text{if}\ \ x\in r(W),\\
0,&\text{if}\ \ x\not\in r(W),\end{cases}
$$
where $r^{-1}$ is the inverse of $r|_W$. Hence, since $\sigma_{\check F}(f)=f$ by~\eqref{eq:sigmaF} and our choice of $F$, we get
\begin{align*}
\|\pi(\sigma_{\check F}(f))\xi\|^2&=\int_{r(W)}D(r^{-1}(x))^{-1}|f(r^{-1}(x))|^2\|\xi_{s(r^{-1}(x))}\|^2d\nu(x)\\
&=\int_{s(W)}|f(s^{-1}(x))|^2\|\xi_x\|^2d\nu(x),
\end{align*}
where $s^{-1}$ is the inverse of $s|_W$. The assumption that $\varphi$ is a ground state implies that the above integral is zero. Since this is true for any $f\in C_c(G)$ supported in $W$, it follows that $\xi_x=0$ for $\nu$-a.e.~$x$ in the neighborhood $s(W)$ of $x_0$. As $x_0\in G^{(0)}\setminus Z$ was arbitrary, we conclude that $\xi_x=0$ for $\nu$-a.e.~$x\in G^{(0)}\setminus Z$. This proves our claim.

As $\xi$ is a unit vector, we see in particular that if there exists a $\sigma^c$-ground state, then $Z\ne\emptyset$.
Next, consider the representation of $G_Z$ on the field $(H_x)_{x\in Z}$ of Hilbert spaces. Integrating it with respect to the measure~$\nu|_Z$ we get a representation $\pi_Z$ of $C^*(G_Z)$. The underlying space $H_Z=\int^\oplus_Z H_xd\nu(x)$ of this representation is a subspace of $H$. By the claim that we proved we have $\xi\in H_Z$. The vector~$\xi$ defines a state~$\psi$ on~$C^*(G_Z)$. As $\varphi(f)=(\pi(f)\xi,\xi)=\int_Z((\pi(f)\xi)_x,\xi_x)d\nu(x)$, it is clear from~\eqref{eq:group-rep} that $\varphi(f)=\psi(f|_{G_Z})$ for all $f\in C_c(G)$. Note that $\psi$ is the only state on~$C^*(G_Z)$ for which this identity can hold, since any function $f\in C_c(G_Z)$ extends to a compactly supported continuous function on $G$.

\smallskip

Conversely, assume $Z\ne\emptyset$ and take a state $\psi$ on $C^*(G_Z)$. We want to show that the linear functional $\varphi_\psi$ on $C_c(G)$ defined by $\varphi_\psi(f)=\psi(f|_{G_Z})$ extends (necessarily uniquely) to a $\sigma^c$-ground state on $C^*(G)$. For this we induce the GNS-representation of $C^*(G_Z)$ to a representation of $C^*(G)$.

Induced representations of groupoid algebras have been studied by a number of authors, see~\cite{ren3} for an overview. The construction essentially goes back to~\cite{mrw}. Consider the space $Y=s^{-1}(Z)\subset G$. At the level of groupoids the induction is done through the commuting actions $G\curvearrowright Y\curvearrowleft G_Z$ arising from the actions of $G$ on itself by left and right translations. These actions define a $C^*(G)$-$C^*(G_Z)$-correspondence $\hh$, see~\cite{ren3} and references therein. Namely, the correspondence $\hh$ is the completion of $C_c(Y)$ with respect to the $C^*(G_Z)$-valued inner product
$$
\langle\zeta,\eta\rangle(g)=\sum_{y\in G_{r(g)}}\overline{\zeta(y)}\eta(yg).
$$
The left and right actions of $C^*(G)$ and $C^*(G_Z)$ are given by the usual convolution operators:
$$
(f\zeta)(y)=\sum_{h\in G^{r(y)}}f(h)\zeta(h^{-1}y)\ \ \text{for}\ \ f\in C_c(G),\ \
(\zeta f_Z)(y)=\sum_{g\in G^{s(y)}_Z}\zeta(yg)f_Z(g^{-1})\ \ \text{for}\ \ f_Z\in C_c(G_Z).
$$

For every $t\in\R$ we define an operator $U_t$ on $C_c(Y)$ by $(U_t\zeta)(h)=e^{it
c(h)}\zeta(h)$. Since the cocycle~$c$ is zero on $G_Z$, these operators preserve the inner product and commute with the right action of $C_c(G_Z)\subset C^*(G_Z)$. Hence they extend to a one-parameter group of unitary operators on the right
$C^*$-Hilbert $C^*(G_Z)$-module $\mathcal H$. It is also clear from the definition that they implement the dynamics~$\sigma^c$:
$$
U_ta\xi=\sigma^c_t(a)U_t\xi\ \ \text{for}\ \ a\in C^*(G)\ \ \text{and}\ \ \xi\in\hh.
$$

Consider the GNS-triple $(H_\psi,\pi_\psi,\xi_\psi)$ defined by the state $\psi$ on $C^*(G_Z)$. Then, using the $C^*$-correspon\-dence~$\hh$, we can induce the representation $\pi_\psi$ to a representation $\pi$ of $C^*(G)$ on $H=\hh\otimes_{C^*(G_Z)}H_\psi$.

In order to define $\varphi_\psi$ assume first that $Z\subset G^{(0)}$ is compact, then its characteristic function
$\un_Z$ is in $C_c(Y)$, where we regard $Z$ a subset of $Y$; note that $Z$ is open $Y$, since $Z=Y\cap G^{(0)}$. In this case
$\xi=\un_Z\otimes\xi_\psi\in H$ is a unit vector and for each $f\in C_c(G)$, we have $\pi(f)\xi=f|_Y\otimes\xi_\psi$.
Hence
$$
(\pi(f)\xi,\xi)=\psi(\langle \un_Z,f|_Y\rangle)=\psi(f|_{G_Z}),
$$
so $\xi$ defines the required state $\varphi_\psi$.

It remains to check that $\varphi_\psi$ is a $\sigma^c$-ground state. The dynamics~$\sigma^c$ is implemented by the unitaries $W_t=U_t\otimes 1$ on $H$. Since $W_t\xi=\xi$, it follows that in order to check that~$\varphi_\psi$ is a $\sigma^c$-ground state we have to show that the operators
$$
W_{\check F}=\int_\R\check F(t)W_tdt
$$
are zero on the vectors $\pi(a)\xi$, $a\in C^*(G)$, for all smooth compactly supported functions $F$ with support in $(-\infty,0)$. But this is clear, since for any $f\in C_c(G)$ we have, similarly to~\eqref{eq:sigmaF}, that
$$
W_{\check F}\pi(f)\xi=(F\circ c)|_Yf|_Y\otimes\xi_\psi,
$$
which is zero as $c\ge0$ on $Y$. This finishes the proof in the case when $Z$ is compact.

When $Z$ is only locally compact, for every $\rho\in  C_c(Z)$, $0\le\rho\le1$, consider the vector $\xi_\rho=\rho\otimes\xi_\psi\in H$ of norm $\le1$. It defines a positive linear functional $\theta_\rho$ of norm $\le1$ on $C^*(G)$. This functional is $\sigma^c$-ground by the same reasoning as above. We have
$$
\theta_\rho(f)=\psi((\rho\circ r)|_{G_Z}f|_{G_Z}(\rho\circ s)|_{G_Z})\ \ \text{for}\ \ f\in C_c(G).
$$
It follows that the required state $\varphi_\psi$ is the limit of the net $(\theta_\rho)_\rho$ in the weak$^*$ topology. It is $\sigma^c$-ground since every functional $\theta_\rho$ is.
\ep

\begin{remark} We would like to draw attention at this point to a minor but potentially misleading omission in \cite{nes}.
The formula analogous to \eqref{eq:group-rep} used in the proof of \cite[Theorem~1.1]{nes} should have included the factor $D(g)^{-1/2}$. This omission has essentially no consequences for the proof of \cite[Theorem~1.1]{nes} and that result holds as stated; a corrected proof can be found in \href{https://arxiv.org/abs/1106.5912v3}{arXiv:1106.5912v3}.
\end{remark}

\begin{remark}
The proof of the theorem implies that there is a unique completely positive contraction $P\colon C^*(G)\to C^*(G_Z)$ such that $P(f)=f|_{G_Z}$ for $f\in C_c(G)$. Indeed, using the $C^*(G)$-$C^*(G_Z)$-correspondence $\mathcal H$ defined in the proof, this map is given by
$$
P(a)=\langle \un_Z,a\un_Z\rangle,\ \ a\in C^*(G),
$$
when $Z$ is compact, and it is given by the pointwise norm limit of the maps $a\mapsto \langle \rho,a\rho\rangle$, where $\rho\in C_c(Z)$, $0\le\rho\le1$, when $Z$ is only locally compact.
\end{remark}

\begin{remark}
For graph $C^*$-algebras, or more precisely, for their reductions by the projections corresponding to vertices, the above theorem recovers a result of Thomsen~\cite[Corollary~5.4]{thom}. Such a reduction $1_vC^*(E)1_v$ is defined by a groupoid $E$ with the unit space consisting of paths in a graph~$\mathcal G$ starting at a vertex $v$. Thomsen considers a cocycle defined by a function $F$ on the edges and describes the ground states of $1_vC^*(E)1_v$ in terms of so called $F$-geodesics. It is not difficult to see that a path starting at $v$ is an $F$-geodesic if and only if it belongs to the boundary set of the restricted cocycle. (As a side remark, if we considered the entire groupoid defining $C^*(E)$, then the boundary set of the corresponding cocycle would only be a subset of the $F$-geodesics.)
\end{remark}

Recall that, given a $C^*$-dynamical system $(A,\sigma)$, a state on $A$ is called a $\sigma$-KMS$_\infty$-state if it is the limit in the weak$^*$ topology of a net of states $(\varphi_i)_i$ such that each $\varphi_i$ is $\sigma$-KMS$_{\beta_i}$ and $\beta_i\to+\infty$, cf.~\cite{con-mar}. Any KMS$_\infty$-state is a ground state by \cite[Proposition~5.3.23]{bra-rob2}, but the converse does not necessarily hold. This is clear from the drastic example of $\mathcal O_\infty$ on which the gauge action has a ground state and no KMS$_\beta$-states for any $\beta\in\R\cup\{\infty\}$,  but also, more subtly, from recent examples that have an abundance of KMS$_\beta$-states for large $\beta$, \cite{LR,CDL}.  The following corollary of Theorem~\ref{thm:ground} makes this phenomenon particularly transparent.

\begin{corollary}\label{cor:KMSinfty}
In the setting of Theorem~\ref{thm:ground}, suppose that $Z \neq \emptyset$ and the groupoid $c^{-1}(0)$ is \'{e}tale. Let~$\varphi_\psi$ be the $\sigma^c$-ground state on $C^*(G)$ corresponding to the state $\psi$ of $C^*(G_Z)$. If $\varphi_\psi$ is a $\sigma^c$-KMS$_\infty$-state, then $\psi$ is tracial.

In particular, if $C^*(G_Z)$ is noncommutative, then the set of $\sigma^c$-KMS$_\infty$-states on $C^*(G)$ is strictly smaller than the set of $\sigma^c$-ground states.
\end{corollary}

\bp
Let $\psi$ be a state of $C^*(G_Z)$ and assume the associated ground  state $\varphi_\psi$ is a $\sigma^c$-KMS$_\infty$-state. Then its  restriction to the subalgebra $C^*(c^{-1}(0))\subset C^*(G)$ is tracial. Since $G_Z$ is the reduction of~$c^{-1}(0)$ by the closed $c^{-1}(0)$-invariant set $Z$, the map $C_c(c^{-1}(0))\ni f\mapsto f|_{G_Z}$ extends to a homomorphism $\pi\colon C^*(c^{-1}(0))\to C^*(G_Z)$. As the restriction of $\varphi_\psi$ to $C^*(c^{-1}(0))$ equals $\psi\circ\pi$, it follows that $\psi$ is tracial.
\ep

Next we consider a class of $C^*$-dynamical systems arising from the reduction of a transformation groupoid to a subset that is not invariant. These include the systems of number theoretic origin based on Hecke algebras that constitute our main application in the next section.

Suppose that $\Gamma$ is a countable discrete
group acting on a second countable, locally compact topological space $X$.
The associated transformation groupoid is the set $\Gamma\times X$ endowed with the product given by $(\gamma',\gamma x)(\gamma,x)=(\gamma'\gamma,x)$; its unit space is naturally homeomorphic to $X$ and its $C^*$-algebra  is the crossed product $\Gamma \ltimes C_0(X)$, which we write with the group on the left for consistency with the groupoid notation. We denote by $\gamma \mapsto u_\gamma$ the canonical embedding of~$\Gamma$ as a group of unitaries into the multiplier algebra of $\Gamma \ltimes C_0(X)$,  and we view the canonical embedding of $C_0(X)$ simply as  inclusion. Then the products $u_\gamma f$ for $\gamma \in \Gamma$ and $f\in C_0(X)$ span a dense $*$-subalgebra of $\Gamma \ltimes C_0(X)$. There is a canonical (dual) coaction of $\Gamma$ on $\Gamma \ltimes C_0(X)$ given by $ u_\gamma f \mapsto u_\gamma f \otimes \lambda_\gamma$, where $\lambda_\gamma$ are the generators of the group $C^*$-algebra of $\Gamma$. For each multiplicative homomorphism $N: \Gamma \to \R^*_+$, there is a quotient coaction of $\R^*_+$ that can be interpreted as a regular action $\sigma^N$ of the additive group~$\R$, viewed as the dual of $\R^*_+$;  it is determined by $$\sigma_t^N(u_\gamma f) = N(\gamma)^{it} u_\gamma f.$$
This dynamics on $C^*(\Gamma\times X)=\Gamma \ltimes C_0(X)$ is the one induced by the $1$-cocycle $c$ on $\Gamma\times X$ given by
$$
c(\gamma,x) := \log N(\gamma), \qquad(\gamma, x) \in \Gamma\times X.
$$

Suppose now that $Y$ is a clopen subset of $X$ such that $\Gamma Y = X$ and form
the reduction of $\Gamma\times X$ by $Y$. By definition, this reduction is the groupoid
$$
\Gamma\boxtimes Y:=\{(\gamma,y)\in\Gamma\times Y\mid\gamma y\in Y\}
$$
in which the product is the obvious restriction of that on $\Gamma \times X$.
The $C^*$-algebra of this restricted groupoid  is canonically isomorphic to the corner
$\un_Y(\Gamma\ltimes C_0(X)) \un_Y$ of the crossed product $\Gamma\ltimes C_0(X)$.
At the level of $C^*(\Gamma \boxtimes Y)$, the restricted dynamics is determined by the restricted cocycle,
\[
\sigma^c_t( f ) (\gamma, y) = N(\gamma)^{it}f(\gamma,y), \qquad  f\in C_c(\Gamma \boxtimes Y)\ \text{and}\ (\gamma,y)\in \Gamma \boxtimes Y.
\]

The first step towards the explicit description of  the  $\sigma^c$-ground states of $C^*(\Gamma\boxtimes Y)$ is to identify the boundary groupoid in the present situation.

\begin{theorem}\label{boundaryandgroundforc}
Suppose $\Gamma$ is a countable discrete group acting on a locally compact second countable space $X$, and $Y$ is a clopen subset of $X$ such that $\Gamma Y = X$. Assume $N: \Gamma \to (0,\infty)$ is a  multiplicative homomorphism and  let $c(\gamma, y)  = \log N(\gamma)$ for $\gamma \in \Gamma $ and $y\in Y$
be the associated 1-cocycle on $\Gamma\boxtimes Y$.
Then
\begin{enumerate}
\item  the boundary set of $c$ is
$$
Y_0 :=
Y\setminus \Big(\bigcup_{\gamma\colon N(\gamma)>1}\gamma Y \Big);
$$
\item the boundary groupoid $\Gamma\boxtimes Y_0$ coincides with $(\ker N)\boxtimes Y_0 $;
\item both $c\inv (0)=(\ker N)\boxtimes Y$ and  $(\ker N)\boxtimes Y_0 $  are \'{e}tale groupoids;
\item if  $Y_0\neq \emptyset$, then the map $\psi \mapsto \varphi_\psi$ defined by
\[
\varphi_\psi(f)=\psi(f|_{(\ker N)\boxtimes Y_0})\ \ \text{for}\ \ f\in C_c(\Gamma \boxtimes Y)
\]
as in \thmref{thm:ground}
is an affine isomorphism of the state space of $C^*((\ker N)\boxtimes Y_0)$ onto the $\sigma^c$-ground state space of $C^*(\Gamma\boxtimes Y)$;
\item  if $Y_0 = \emptyset$,  there are no $\sigma^c$-ground states on $C^*(\Gamma\boxtimes Y)$.
 \end{enumerate}
 \end{theorem}
 \begin{proof}
First notice that, by definition, the boundary set of $c$ is
$$
\{y\in Y:\ \text{if}\ y=\gamma z\in Y\ \text{for some}\ \gamma\in\Gamma\ \text{and}\ z\in Y,\ \text{then}\ N(\gamma)\le1\},
$$
so part (1) is clear.
If $ \gamma Y_0\cap Y_0\ne\emptyset$ we also have $  Y_0\cap \gamma\inv Y_0\ne\emptyset$  hence
$N(\gamma) \geq1$ and also $N(\gamma\inv) \geq 1$ so $N(\gamma) =1$.
Thus the reduction of $\Gamma \boxtimes Y$ by $Y_0$  is $(\ker N)\boxtimes Y_0$, proving part (2).
The groupoid $c^{-1}(0)$  is  \'{e}tale  because
$Y$ is clopen and $\Gamma$ is discrete, and the rest of part (3)
follows from \lemref{kernellemma}.
Since the boundary groupoid of $c$ is $(\ker N)\boxtimes Y_0 $,
parts (4) and (5) follow from Theorem~\ref{thm:ground}.
\end{proof}

\begin{lemma}\label{Sdefinition}
In the situation of \thmref{boundaryandgroundforc} let
\[
S :=\{\gamma\in\Gamma: \gamma Y_0\cap Y\neq \emptyset\}
\quad
\text{ and }
\quad
S_0:=\{\gamma \in \Gamma \mid \gamma Y_0\cap Y_0\neq \emptyset\}.
\]
Then $S_0 \subset S \cap S\inv \subset \ker N$ and $N(\gamma) \geq 1$ for every $\gamma \in S$.
If in addition  the series
\begin{equation}\label{eq:critical-beta}
\zeta_S(\beta):=\sum_{s\in S}N(s)^{-\beta}
\end{equation}
has a finite abscissa of convergence, then $(\ker N) \cap S$ is a finite set.
 \end{lemma}
\begin{proof}
The chain of inclusions is part of the proof of \thmref{boundaryandgroundforc}, the rest follows from elementary considerations about Dirichlet series.
\end{proof}

\begin{remark}\label{remarkS_0finite}
Clearly we have the equality $(\ker N) \boxtimes Y_0 = S_0 \boxtimes Y_0$ of sets. Therefore, if $S_0$ is finite, then the orbit and the isotropy group of every $y_0\in Y_0$ under the action of the groupoid $(\ker N) \boxtimes Y_0$ has at most $|S_0|$ elements each, and thus  $C^*(\Gamma \boxtimes Y_0)=C^*((\ker N) \boxtimes Y_0)$ is a subhomogeneous  $C^*$-algebra.
\end{remark}

Next we consider a set-up in which the characterization of KMS$_\beta$-states from \thmref{thm:KMS}
has an explicit interpretation in terms of traces on the $C^*$-algebra of the boundary groupoid.

\begin{theorem}[{cf.~\cite[Proposition~1.2]{LLN}}]\label{KMSandgroundforY}
Let $X$, $Y$, $\Gamma$, $N$, $c$ and $Y_0$ be as in \thmref{boundaryandgroundforc}. Suppose that $Y_0\neq \emptyset$, let
$$
S :=\{\gamma\in\Gamma: \gamma Y_0\cap Y\neq \emptyset\}
$$
be defined as in \lemref{Sdefinition} and suppose that the series $\zeta_S(\beta) $ has a finite abscissa of convergence $\beta_0 <\infty$.
Assume that there are Borel subsets $Y_n\subset Y$ and elements $\gamma_n\in\Gamma$ for $n=1,2,\dots$ such that
\begin{enumerate}
\item[(i)] for every neighborhood $U$ of $Y_0$ we have $Y\setminus SU\subset\cup^\infty_{n=1}Y_n$;
\item[(ii)] $N(\gamma_n)\ne1$ and $\gamma_nY_n=Y_n$ for all $n\ge1$.
\end{enumerate}
Then, for each $\beta\in(\beta_0,+\infty]$, there is  an affine isomorphism of the simplex of tracial states on $C^*((\ker N)\boxtimes Y_0)$
onto the  simplex of $\sigma^c$-KMS$_\beta$-states on $C^*(\Gamma\boxtimes Y)$.
\end{theorem}

\bp
We consider first the case $\beta\in(\beta_0,+\infty)$.
By Theorem~\ref{thm:KMS} we know that in order to describe the KMS$_\beta$-states we first of all have to describe all quasi-invariant Borel probability measures on~$Y$ with Radon--Nikodym cocycle $e^{-\beta c}$, that is, measures $\mu$ such that if $A\subset Y$ is Borel and $\gamma A\subset Y$ for some $\gamma\in\Gamma$, then
\begin{equation}\label{eq:RN}
\mu(\gamma A)=N(\gamma)^{-\beta}\mu(A).
\end{equation}
We claim that for any such measure $\mu$ we have $\mu(Y_0)>0$ and the map $\mu\mapsto\mu(Y_0)^{-1}\mu|_{Y_0}$ is an affine isomorphism between the set of all such measures and the set of Borel probability measures on $Y_0$ invariant under the partial action of $\ker N$.

By assumption (ii) and property~\eqref{eq:RN}, we have $\mu(Y_n)=0$ for all $n\ge1$. Then assumption (i) implies that $SU\cap Y$ is a subset of $Y$ of full measure for any open neighborhood $U$ of $Y_0$ in $Y$. Hence
$$
1=\mu(SU\cap Y)\le\sum_{s\in S}\mu(sU\cap Y)\le\sum_{s\in S}N(s)^{-\beta}\mu(U)=\zeta_S(\beta)\mu(U).
$$
It follows that $\mu(Y_0)\ge\zeta_S(\beta)^{-1}$, which proves the first part of our claim.

Next, let us show that $\mu$ is concentrated on $\Gamma Y_0\cap Y=SY_0\cap Y$. Indeed, if this were not the case, the measure $\tilde\mu$ defined by
$$
\tilde\mu(A)=\mu(Y\setminus\Gamma Y_0)^{-1}\mu(A\setminus\Gamma Y_0)
$$
would be a quasi-invariant Borel probability measure with Radon--Nikodym cocycle $e^{-\beta c}$.  By what we have just proved  this would mean $\tilde\mu(Y_0)>0$, contradicting $\tilde\mu(Y_0)=0$, which
holds by construction.

It follows that $\mu$ is completely determined by its restriction to $Y_0$, which is a measure invariant under the partial action of $\ker N$. To finish the proof of the claim it remains to show that any $(\ker N)$-invariant probability measure on $Y_0$ arises from a probability measure $\mu$ satisfying \eqref{eq:RN}. Let $\nu$ be such a measure on $Y_0$. We can choose distinct elements $\gamma_i\in\Gamma$ and Borel subsets $Z_i\subset Y_0$ such that $\Gamma Y_0$ is the disjoint union of the sets~$\gamma_i Z_i$. Define a Borel measure~$\tilde\nu_\beta$ on~$X$~by
$$
\tilde\nu_\beta(A)=\sum_iN(\gamma_i)^{-\beta}\nu(\gamma_i^{-1}A\cap Z_i).
$$
As in the proof of~\cite[Lemma~2.2]{LLN-GL2}, one sees that this measure does not depend on the choice
of $\gamma_i$ and $Z_i$, satisfies~\eqref{eq:RN} for all Borel subsets $A\subset X$ and all $\gamma\in\Gamma$, and $\tilde\nu_\beta|_{Y_0}=\nu$. Furthermore, if $\gamma_i^{-1}Y\cap Z_i\ne\emptyset$ for some $i$, then $\gamma_i\in S$, which implies that
$$
\tilde\nu_\beta(Y)\le\sum_{i\colon\gamma_i\in S}N(\gamma_i)^{-\beta}\le\zeta_S(\beta).
$$
Therefore $\mu_\beta=\tilde\nu_\beta(Y)^{-1}\tilde\nu_\beta|_Y$ is the required measure on $Y$. This completes the proof of the claim.

\smallskip

Later in the proof we will need the following consequence of the above considerations. By the $c^{-1}(0)$-invariance of the boundary set, an element $\gamma\in\Gamma$ can map a point in $Y_0$ into $Y\setminus Y_0$ only if $N(\gamma)>1$. It follows that $\tilde\nu_\beta(Y\setminus Y_0)\le\sum_{s\in S\colon N(s)>1}N(s)^{-\beta}$, whence
$$
\mu_\beta(Y_0)\ge\left(1+\sum_{s\in S\colon N(s)>1}N(s)^{-\beta}\right)^{-1}.
$$
Hence $\mu_\beta(Y_0)\to1$ as $\beta\to+\infty$, and therefore the measures $\mu_\beta$ converge in norm to the measure $\nu$ on $Y_0$ we started with (considered as a measure on $Y$).

\smallskip

Let us now complete the classification of KMS$_\beta$-states for $\beta\in(\beta_0,+\infty)$. By Theorem~\ref{thm:KMS}, every such state~$\varphi$ is given by a probability measure $\mu$ on $Y$ satisfying~\eqref{eq:RN} and a $\mu$-measurable field of tracial states~$\tau_y$ on the $C^*$-algebras of the stabilizers $\Gamma_y\subset\Gamma$ of points $y\in Y$ such that
\begin{equation}\label{eq:trace}
\tau_y=\tau_{\gamma y}(u_\gamma\cdot u_\gamma^{-1})
\end{equation}
whenever $y\in Y$ and $\gamma y\in Y$. As we showed above, the measure $\mu$ is concentrated on $\Gamma Y_0\cap Y$ and is completely determined by its restriction to $Y_0$. Condition~\eqref{eq:trace} shows then that, modulo a set of $\mu$-measure zero, the field $(\tau_y)_{y\in Y}$ is also determined by its restriction to~$Y_0$.
By Theorem~\ref{thm:KMS}, the measure $\nu=\mu(Y_0)^{-1}\mu|_{Y_0}$ together with the field $(\tau_y)_{y\in Y_0}$ define a tracial state on $C^*((\ker N)\boxtimes Y_0)$. We therefore get an injective affine map from the set of KMS$_\beta$-states on $C^*(\Gamma\boxtimes Y)$ into the set of tracial states on $C^*((\ker N)\boxtimes Y_0)$.

Conversely, if we now start with a tracial state on $C^*((\ker N)\boxtimes Y_0)$, we see that, again by Theorem~\ref{thm:KMS}, it is given by a probability measure $\nu$ on $Y_0$ invariant under the partial action of $\ker N$ and a field of tracial states~$\tau_y$ on~$C^*(\Gamma_y)$, $y\in Y_0$, satisfying~\eqref{eq:trace} whenever $y\in Y_0$ and $\gamma y\in Y_0$. As we showed above, the measure $\nu$ arises from a probability measure $\mu$ on $Y$ satisfying~\eqref{eq:RN}. The field $(\tau_y)_{y\in Y_0}$ extends in a unique way to a field of tracial states $\tau_y$, $y\in \Gamma Y_0$, satisfying~\eqref{eq:trace}. As $\mu(Y\setminus\Gamma Y_0)=0$, the pair $(\mu,(\tau_y)_{y\in Y\cap\Gamma Y_0})$ defines a KMS$_\beta$-state of $C^*(\Gamma\boxtimes Y)$. This finishes the proof of the result in the case of finite $\beta$.

\smallskip

Turning to the remaining case of KMS$_\infty$-states, take a tracial state $\tau$ on $C^*((\ker N)\boxtimes Y_0)$. For each $\beta\in(\beta_0,+\infty)$, let $\varphi_\beta$ be the KMS$_\beta$-state corresponding to $\tau$. As we observed above, the measures~$\mu_\beta$ on~$Y$ defined by $\varphi_\beta|_{C_0(Y)}$ converge in norm as $\beta\to+\infty$ to the measure $\nu$ on $Y_0$ defined by $\tau|_{C_0(Y_0)}$. From the construction of $\varphi_\beta$ it follows then that the states $\varphi_\beta$ converge in norm to a KMS$_\infty$-state $\varphi$ such that $\varphi(f)=\tau(f|_{(\ker N)\boxtimes Y})$ for $f\in C_c(\Gamma\boxtimes Y)$. But by Corollary~\ref{cor:KMSinfty} we know that any KMS$_\infty$-state has this form for some tracial state $\tau$. Hence we get the required affine isomorphism.
\ep

\section{Hecke algebras of orientation preserving affine groups}\label{KMShecke}

By a Hecke pair $(G,\Gamma)$ we mean a group $G$ with a subgroup $\Gamma$ such that every double coset $\Gamma g\Gamma$ contains finitely many left and right cosets of $\Gamma$.
The Hecke algebra $\hh(G,\Gamma)$ of such a pair is defined as the vector space of  complex valued $\Gamma$-bi-invariant
functions on $G$ supported on finitely many double cosets,
endowed with the product
$$
(f_1\ast f_2)(g)
=\sum_{h\in G/\Gamma} f_1(h)f_2(h^{-1}g).
$$
This construction of course also makes sense for functions with values in any field $K$, in which case we denote the Hecke algebra by $\hh_K(G,\Gamma)$.
For $g\in G$, we denote by $[\Gamma g\Gamma]\in \hh_K(G,\Gamma)$ the characteristic function of the double coset $\Gamma g\Gamma$.

When $K=\C$, we in addition have an involution on $\hh(G,\Gamma)$ defined by
$$
f^*(g)=\overline{f(g^{-1})}.
$$
Then the formula
$$
(\lambda(f)\xi)(g)=\sum_{h\in G/\Gamma}f(h)\xi(h^{-1}g)
$$
defines a $*$-representation $\lambda$ of $\hh(G,\Gamma)$ on $\ell^2(\Gamma\backslash G)$. The \emph{Hecke $C^*$-algebra} of $(G,\Gamma)$ is defined as the norm closure of $\lambda(\hh(G,\Gamma))\subset B(\ell^2(\Gamma\backslash G))$ and is denoted by $C^*_r(G,\Gamma)$, cf.~\cite{bos-con}. We remark that in general $\hh(G,\Gamma)$ does not admit a universal enveloping $C^*$-algebra.

\medskip

Consider now a number field $K$ with ring of integers $\OO$. An element of $K$ is called {\em totally positive} if it is positive in every real embedding of $K$. Let $\kps$ be the multiplicative group of totally positive elements, with $\op^\times := \ox\cap \kps$ the monoid of totally positive algebraic integers in $\kk$, and  $\ops := \os \cap \kps$ the group of totally positive units in $\oo$. Following \cite{bos-con,LNT} we consider the pair
$$
\po:=\left(\begin{array}{ll}1&{\OO}\\0&{{\OO}_+^*}\end{array}\right)\ \  \subset \ \ \pk:=\left(\begin{array}{ll}1&{K}\\0&{{K}_+^*}\end{array}\right).
$$
That this is a Hecke pair of groups can be seen by embedding it into a topological pair, see \cite[Section~2]{LNT}, and also by a direct approach that counts the left cosets in every double coset, along the lines of \cite{bos-con,LvF}.  Both $\pk$ and $\po$ are groups of affine transformations of $K$ that preserve the orientation in every real embedding, so we will refer to them as the \emph{orientation preserving Hecke pair of affine groups} associated to $K$.
The associated Hecke $C^*$-algebra $\heck$ has a universal property with respect to $*$-representations of $\hh(\pk, \po)$,
and also a presentation in terms of generators and relations arising from two classes of double cosets which we discuss briefly next.

Following \cite{bos-con} we consider distinguished elements in $\heck$ given by two specific collections of double cosets. For each $a\in \opx$ let $N_a = |\oo /a\oo|$ be the absolute norm of $a$ and define
\begin{gather}\label{E: mua}
\mu_a:= \frac{1}{\sqrt{N_a}}\left[\po\matr{0}{a}\po\right]. 
\end{gather}
For each $r\in \kk$ let  $R(r)$ be the
number of right cosets in the double coset of $\smatr{r}{1}$, and  define
\begin{gather}\label{E: er}
e_r:= \frac{1}{R(r)}\left[\po\matr{r}{1} \po\right] .
\end{gather}
An argument similar to the proof of \cite[Lemma 1.2]{LvF} shows that $R(r)$ is equal to the cardinality of the $\ops$-orbit of $r$ modulo $\oo$, or, equivalently, the index of the subgroup
${\ops}_r := \{u\in \ops: ur = r \pmod \oo\}$ in $\ops$; see also \cite[Proposition 1.3]{llnfinpart}.
\begin{proposition}\label{Heckepresentation}
The elements $\mu_a$  and $e_r$ satisfy the following relations:
\begin{align}
&\mu_w=1, \qquad w\in \ops ;\label{mufirst}\\
&\mu_a^*\mu_a^{}=1, \qquad a\in \opx ;\\
 &\mu_a\mu_b=\mu_{ab}, \qquad a,b \in \opx ;\label{mulast}\\
&  e_{wr + b} = e_r   , \qquad r\in \kk, \  w\in \ops,\ b\in \oo ; \label{tetamid}\\
&e_0 = 1 ; \label{tetafirst} \\
&   e_r^*  = e_{-r}, \qquad r\in \kk ; \\
&e_r e_s = \frac{1}{R(r)} \;
\frac{1}{R(s)}\sum_{u\in\ops \!/\stabop{r}}\sum_{v\in\ops
\!/\stabop{s}}e_{ur+vs},
  \qquad r,s\in \kk, \label{tetalast}
\end{align}
and
\begin{gather}\label{covariance}
\mu_a^{} e_r \mu_a^*
= \frac{1}{N_a}\sum_{b\in\oo\!/a\oo}e_{\frac{r+b}{a}},
\qquad a\in \opx, \ r\in \kmo.
\end{gather}
Moreover, the above relations give a presentation of $\hh(\pk,\po)$ as a $*$-algebra, and of $\heck$ as a $C^*$-algebra.
\end{proposition}
The proof goes along the same lines as the proofs of Proposition 1.7 and Theorem 1.10 in \cite{LvF},
with the minor difference that since the unit $-1$ is not totally positive, the generators $e_r$ are not selfadjoint (for $r\ne0$) in contrast to the situation studied in \cite{LvF}.

It is immediate from the relations that $\mu_a$ depends only on $a$ modulo $\ops$,
while $e_r$ depends only on the $\ops$-orbit of $r+\oo \in \kmo$.
Thus,  $\mu$ gives a representation of $\opx/\ops$ by isometries and
the multiplication rule of the $e_r$ reflects the multiplication
 of $\ops$-averages of generating unitaries  in $C^*(\kmo)$.
By~\eqref{tetalast}, the linear span of $\{e_r: r\in \kk\}$ is a commutative, unital $*$-subalgebra of\/ $\mathcal H(\pk,\po)$, which happens to be universal for the relations \eqref{tetafirst}-\eqref{tetalast}
in the category of $*$-algebras.
The closure of $\lspa \{e_r: r\in \kk\}$ inside $\heck$ is a commutative $C^*$-algebra and
is universal, for the same relations, in the category of $C^*$-algebras,
see~\cite[Proposition~1.6]{LvF} for a similar argument.

The usual time evolution on the $C^*$-algebra of a Hecke pair is defined in terms of the modular homomorphism
$$
\Delta(g) :=\frac{|\po\backslash \po g\po|}{|\po g \po/\po|},\qquad g\in \pk,
$$
which for our Hecke pair is given explicitly by
\[
\Delta\begin{pmatrix}1 & y\\ 0 & x\end{pmatrix}=N_K(x)^{-1},\qquad y\in K \text{ and } x\in K^*_+,
\]
where $N_K(x)$ is the absolute norm of the principal fractional ideal $(x)$, so $N_K(a)=N_a=|\oo /a\oo|$ if $a\in\OO^\times$.
Accordingly, we define a time evolution on $\heck$ by
\begin{equation}\label{eq:modular-function}
\sigma_t([\po g \po]) :=\Delta(g)^{-it}[\po g \po]\ \ \text{for}\ \ g\in  \pk.
\end{equation}
On generators this amounts to
 \[
 \sigma_t(\mu_a) = N_a^{-it}\mu_a \quad\text{ and } \quad \sigma_t(e_r) = e_r.
 \]

The $C^*$-algebra $\heck$ can be written as a semigroup crossed product, which, in turn, can be realized as the $C^*$-algebra of a reduction of a transformation groupoid, see~\cite[Proposition~2.2]{LNT}. This will allow us to compute KMS-states and ground states using the techniques from \secref{groupoids}.

The right hand side of \eqref{covariance} defines an action $\alpha$ of the semigroup $\opx/\ops$
by injective endomorphisms  of $C^*(e_r: r \in \kk)$ given~by
 \[
 \alpha_a ( e_r) :=  \frac{1}{N_a}\sum_{b\in\oo\!/a\oo}e_{\frac{r+b}{a}}.
 \]
That this is indeed an action by endomorphisms can be proved directly using the presentation of  $C^*(e_r: r \in \kk)$,
but is actually obvious if we use the left-hand side of \eqref{covariance}.
A standard argument using universal properties now shows that $\heck$ is canonically isomorphic to the semigroup crossed product
$\opx/\ops \ltimes C^*(e_r: r\in \kk)$.

The spectrum of $C^*(e_r: r\in \kk)$ is most naturally described in terms of the ring $\akf$ of finite adeles associated with $K$. We denote by $\hat\OO\subset\akf$ the compact subring of integral adeles. Viewed as an abelian group, $\hatok$ can be identified with the Pontryagin dual of $\kmo$, although there is no such canonical identification. Consider the closure $\opsbar$  of $\ops$ inside the integral ideles $\hatok^*$. It is a compact group under multiplication, which also coincides with the natural profinite compactification of $\ops$ obtained from its action on $\kmo$. Then $C^*(e_r: r\in \kk) \cong C(\hatok/\opsbar)$, where $\hatok/\opsbar$ denotes the orbit space for the action of $\opsbar$ on $\hatok$ by multiplication.

The group $\kps$ of totally positive elements in $\kk$ acts
by multiplication on $\akf$.
Let  $\Gamma := \kps/\ops$. Since the action of $\ops$ on $\akf$ is not trivial, multiplication does not induce an action of~$\Gamma$  on~$\akf$, but
$\Gamma$ acts by multiplication on the orbit space $X := \akf /\opsbar$,
which is a  totally disconnected, second countable, locally compact Hausdorff space.
This gives rise to the transformation groupoid
\[
\Gamma \times X = \kps/\ops \times \akf/\opsbar,
\]
and we will be interested in its reduction by the subset $Y : = \hatok/\opsbar$, which is the spectrum of $C^*(e_r: r\in \kk)$. The $C^*$-algebra of the groupoid $\kps/\ops \boxtimes \hatok/\opsbar$ coincides with the corner
$$
p_{\hatok/\opsbar} \left(\kps/\ops \ltimes C_0(\akf/\opsbar)\right) p_{\hatok/\opsbar},
$$
where $p_{\hatok/\opsbar}$ denotes the characteristic function of ${\hatok/\opsbar}$. By the dilation/extension results of~\cite{minautex}, this corner is also isomorphic to the semigroup crossed product $\opx/\ops \ltimes C(\hatok/\opsbar)$. The last $C^*$-algebra is isomorphic to
$$
\opx/\ops\,{}_\alpha\!\!\ltimes C^*(e_r: r\in \kk)\cong\heck.
$$

It will be useful to dig a little deeper into the construction of the above isomorphisms in order to identify explicitly the images of the
canonical generators of the Hecke algebra.

\begin{proposition} \label{prop:iso}
Let  $\chi:\akf \to \T$ be a character implementing a self-duality of $\akf$ in which the annihilator of $\hatok$ equals~$\hatok$.
There is an isomorphism
\begin{equation}\label{eq:FK3}
C_r^*(\Pk, \Po)\cong   p_{\hatok/\overline{\OO_+^*}}\bigl(K_+^*/\OO_+^*\ltimes C_0(\akf/\overline{\OO_+^*}) \bigr)p_{\hatok/\overline{\OO_+^*}} = C^*(\kps/\ops\, \boxtimes \,\hatok/\opsbar)
\end{equation}
determined by
\[
 \mu_x \mapsto p_{\hatok/\overline{\OO_+^*}}\lambda_{(x)} p_{\hatok/\overline{\OO_+^*}}, \qquad x\in \opx,
\]
 and
 \[
 e_y \mapsto \frac{1}{R(y)}\int_{\overline{\OO_+^*}}\chi(\cdot\,uy)du\in C(\hat\OO/\overline{\OO_+^*}), \qquad y \in K.
\]
\end{proposition}
\bp
We refer to the proof of \cite[Proposition 2.2]{LNT}. By the arguments there and~\cite[Lemma~2.3]{llnfinpart}, the elements  $\mu_x$ for $x\in \opx$ are  mapped into
$p_{\hatok/\overline{\OO_+^*}}\lambda_{(x)} p_{\hatok/\overline{\OO_+^*}}$. On the other hand, the images of the elements  $R(y)e_y$ for $y\in K$ can be computed by first mapping $y$ into the corresponding generator of $C^*(\kmo)=C^*(\akf/\hat\OO)$, then using an isomorphism $C^*(\kmo)\cong C(\hat\OO)$, and finally by averaging the result by the action of $\overline{\OO_+^*}\subset\ohs$. The isomorphism $C^*(\kmo)\cong C(\hat\OO)$ depends on the choice of $\chi$ and maps the generator $\lambda_{y+\OO}$ of $C^*(\kmo)$ into the function $\chi(\cdot\, y)$.
\ep

\begin{remark}
Note that the integrals $\int_{\overline{\OO_+^*}}\chi(\cdot\,uy)du$ are averages over finite quotients of~$\overline{\OO_+^*}$, so the image of~$R(y)e_y$ is a normalized sum of finitely many characters of $\hatok$ of the form $\chi(\cdot uy)$, $u\in\overline{\OO_+^*}$.
\end{remark}

The dynamics $\sigma$ on $\heck$ corresponds to the dynamics on $C^*(\kps/\ops\, \boxtimes \,\hatok/\opsbar)$ defined by the homomorphism $N_K\colon \kps/\ops\to(0,+\infty)$. Our next goal is to understand the boundary set of the associated cocycle on $\kps/\ops\, \boxtimes \,\hatok/\opsbar$.

Each prime ideal $\pid$ determines a discrete valuation $v= v_\pid$ on $K$, defined to be the integer-valued function
 on the set of integral ideals giving the largest power of $\pid$ that divides a given ideal. This can be extended to a valuation defined on $\hatok$ and on $\akf$ using a uniformizing element $\pi_\pid$ in place of the prime ideal $\pid$, and allowing finitely many finite negative powers. Clearly, discrete valuations factor through the quotient $\hatok/\opsbar$.
Let $\mko$ denote the set of all such discrete valuations. For each integral ideal $\aid$ define
 \[
\Omega_\aid :=  \{\omega \in \hatok/\opsbar : v(\omega) = v(\aid) \text{ for all } v\in \mko \}.
\]
These are mutually disjoint sets such that $\sqcup_\aid \Omega_\aid = (\jkf \cap \hatok)/\opsbar$.

Denote by $\clkp$ the narrow class group of $K$, the quotient of the group of fractional ideals by the subgroup of principal fractional ideals with a totally positive generator.

   \begin{proposition}  \label{disjointcomponents}
For each narrow ideal class $c\in \clkp$, let
$\aid_{c,1}, \aid_{c,2}, \ldots, \aid_{c,k_c}$ be the integral ideals
in $c$ that have minimal norm among all integral ideals in $c$. Then the boundary set $Y_0$ of the
cocycle determined by $N_K$ on $\kps/\ops\boxtimes \hatok/\opsbar$ decomposes as  the disjoint union
\[
Y_0:=(\hatok/{\opsbar})\setminus \bigcup_{x\in K_+^*: \  N_K(x)>1}( x\hatok/\opsbar)
=\bigsqcup_{ c\in \clkp} \, \bigsqcup_{1\leq j \leq k_c}  \Omega_{\aid_{c,j}}.
\]
 \end{proposition}
 \begin{proof}
Recall that we denote $\hatok/\opsbar$ by $Y$. For $x,y\in K^*$, denote by $[x,y]$ the lcm of the fractional ideals $(x)$ and $(y)$.

Let $c$ be a narrow class and let $\omega \in \Omega_{\aid_{c,j}}$ for some $1 \leq j \leq k_c$. If
$\omega \in Y\cap x Y$ for some $x \in \kps$, then $\omega \in \Omega_{\bid [1,x]}$
 for some integral ideal $\bid$, and since the sets $\Omega_{\aid}$ are disjoint,
 we have  $\bid [1,x] = \aid_{c,j}$. But then
 $\aid_{c,j} x\inv= \bid [1,x]  x\inv = \bid [x\inv,1] $
 is an integral ideal in the same narrow class as  $\aid_{c,j}$, so $N_K(x) \leq 1$.
 Hence $\omega \notin x  Y $ for every $x$ with $N_K(x) > 1$, and thus $\omega \in Y_0$.

 Conversely, suppose now $\omega \in Y_0$. We claim that the valuation vector of $\omega$
 has finite support and finite values. Otherwise the  \emph{superideal} corresponding to those valuations
would have infinitely many prime ideal factors (counted with multiplicity) in the same narrow class.
The product of any subcollection of these with cardinality the narrow class number
would produce a totally positive $a \in \opx$ such that $\omega \in a Y $, which would
contradict the assumption that $\omega \in Y_0$ because $N_K(a) >1$. This proves the claim and implies that
 $\omega \in \Omega_\aid$ for an integral ideal $\aid$.

Let $c$ be the narrow class of $\aid$ and
 suppose $\bid $ is an integral ideal in $c$. Then
there exists $x\in \kps$ such that  $\aid = x \bid$, so
we have $\omega \in x Y$. By assumption, $\omega \in Y_0$,
so $N_K(x) \leq1$. This implies that $\aid$ is of minimal norm in its class, that is, $\aid = \aid_{c,j}$ for some $j = 1,\ldots, k_c$. \end{proof}

\begin{corollary}\label{idealratios}
In the context of \proref{disjointcomponents}, the subset $S_0\subset \Gamma$
from \lemref{Sdefinition} is a finite subset $(\kps/\ops)_0$ of $\kpsone/\ops$, where $\kpsone$ is the group of totally positive elements of norm one.
This set consists of principal ideals factoring as the ratio of two ideals of minimal norm
in the same narrow class:
\[
(\kps/\ops)_0 = \{\aid_{c,i}^{} \aid_{c,j}\inv : c\in \clkp \ \  i,j = 1, 2, \ldots, k_c\}.
\]
\end{corollary}

\bp
Suppose $\omega \in Y_0$ and $x \omega \in Y_0$ for some $x \in \kps$. By \proref{disjointcomponents},
$\omega \in \Omega_{\aid_{c,j}}$ and $x\omega \in \Omega_{\aid_{d,i}}$ for some $c,d\in \clkp$ and $i,j$. But then $x\aid_{c,j}=\aid_{d,i}$, so $c=d$ and $(x)=\aid_{c,i}\aid_{c,j}\inv$.

On the other hand, every $\gamma \in \kps/\ops$ that
factors as $\aid_{c,i}^{} \aid_{c,j}\inv$ maps elements of $\Omega_{\aid_{c,j}} \subset Y_0$ to elements
of $\Omega_{\aid_{c,i}} \subset Y_0$, so $\gamma\in (\kps/\ops)_0$.
\ep

\begin{corollary}\label{cor:boundaryalg}
In the context of \proref{disjointcomponents} the boundary groupoid $\kps/\ops \boxtimes Y_0$ has trivial isotropy and
$$
C^*(\kps/\ops \boxtimes Y_0)\cong\left(\bigoplus_{c\in\clkp}\Mat_{k_c}(\C)\right)\otimes C(\ohs/\opsbar).
$$
\end{corollary}

\bp
The groupoid $\kps/\ops \boxtimes Y_0$ has trivial isotropy, because none of the elements in $Y_0$
have the valuation of a zero-divisor. Therefore this groupoid is an equivalence relation on $Y_0$. By the proof of the previous corollary, the equivalence class of every point $\omega \in \Omega_{\aid_{c,j}}$ consists precisely of~$k_c$ elements, with one point in each $\Omega_{\aid_{c,i}}$. It follows that
$$
C^*(\kps/\ops \boxtimes Y_0)\cong\bigoplus_{c\in\clkp}\Mat_{k_c}(C(\Omega_{\aid_{c,1}}))\cong\left(\bigoplus_{c\in\clkp}\Mat_{k_c}(\C)\right)\otimes C(\ohs/\opsbar),
$$
where we used that the sets $\Omega_\aaa$ are all homeomorphic to $\ohs/\opsbar$.
\ep

\begin{theorem}\label{thm:gs-for-redHecke}
Let $K$ be an algebraic number field and consider the time evolution $ \sigma$ on  $C_r^*(\Pk, \Po)$ defined by~\eqref{eq:modular-function}.
Let
$$
Y_0:=(\hatok/{\opsbar})\setminus \bigcup_{x\in K_+^*: \   {N}_K(x)>1}( x\hatok/\opsbar)
$$
be the set described in Proposition~\ref{disjointcomponents}.
Then
\begin{enumerate}
\item there is an affine isomorphism of the ground state space of $(C_r^*(\Pk, \Po), {\sigma})$ onto the state space of $C^*(\kps/\ops\boxtimes Y_0)$;

\item for each $\beta\in(1,+\infty]$, there is an affine isomorphism of the simplex of $\sigma$-KMS$_\beta$-states on the $C^*$-algebra $C_r^*(\Pk, \Po)$ onto the simplex of Borel probability measures on $\clkp\times(\ohs/\opsbar)$.
\end{enumerate}
\end{theorem}

\begin{proof}
Part (1) is an immediate consequence of the isomorphism $C_r^*(\Pk, \Po)\cong C^*(\kps/\ops\boxtimes Y)$ from~\eqref{eq:FK3}, where $Y=\hatok/\opsbar$, and \thmref{boundaryandgroundforc}.

Part (2) for $\beta\in(1,\infty)$ is already proved in \cite[Theorem~2.5]{LNT} (where the interval $(0,1]$ is also dealt with) using the classification of KMS-states on the full Bost-Connes system associated with~$K$. Here we will give a more direct argument based on Theorem~\ref{KMSandgroundforY} and also consider the case $\beta=\infty$.

The first step is to characterize the set $S\subset\Gamma=\kps/\ops$ from \lemref{Sdefinition} and to determine the region of convergence of its zeta function $\zeta_S(\beta)$. For this, suppose $\gamma \in S$; then there exists $\omega_0 \in Y_0$ such that
 $\gamma \omega_0 = \omega \in Y$.
By \proref{disjointcomponents}, $\omega_0 \in \Omega_{\aid_{c,j}}$ for some narrow class $c$ and some $j = 1, 2, \ldots, k_c$, so we have
$v(\omega_0) = v(\aid_{c,j})$ and $v(\gamma) = v(\omega) - v(\aid_{c,j}) \geq - v(\aid_{c,j})$.  Choose $d\in \opx$ such that $(d) \subset \cap_{c,j} \aid_{c,j}$, so that $v(d) \geq v(\aid_{c,j})$ for all $v\in \mko$, proving that $S \subset (\frac{1}{d}\opx/\ops)$. As a consequence $\zeta_S(\beta)$ has the same abscissa of convergence as the  Dedekind zeta function of $K$.

The next step is to come up with a countable collection of Borel subsets of $Y$ and associated group elements satisfying properties (i) and (ii) of \thmref{KMSandgroundforY}. For each $v\in \mko$,  let
$$Y_v:=\{\omega \in Y: v(\omega) = \infty\},$$
and let $x_v$ be a totally positive element whose principal ideal is a power of the prime ideal $\pid$ corresponding to the discrete valuation $v$, e.g., let $h_+$ denote the narrow class number of $\kk$ and choose $x_v$ to be a totally positive generator of the principal ideal $\pid^{h_+}$. Notice that $\bigcup_{v} Y_v$ is the image in $Y$ of the set of zero divisors in $\hatok$. We clearly have $  N_K(x_v) >1$ and  $x_v Y_v = Y_v$ for every $v\in \mko$, so property (ii) in \thmref{KMSandgroundforY} is satisfied.

It remains to prove that the collection $\{Y_v\}_v$ also has  property (i), that is,
for every open set $U$ containing $Y_0$ and every $\omega \in Y\setminus \cup^\infty_{n=1}Y_n$ there exists $\gamma \in S$ such that $\omega \in \gamma U$.
Suppose $U$ is a basic open set containing $Y_0$; we may assume that  $U$ is of the form
\[
U_F = \bigcup_{(c,j)}  U_{c,j} = \bigcup_{(c,j)} \{\omega \in Y:
 v(\omega) = v(\aid_{c,j})   \text{ for } v\in F \},
\]
where $F$ is a finite subset of $\mko$ containing every valuation that does not vanish on $\aid_{c,j}$ for some~$c$ and some $j = 1,2, \ldots, k_c$.

Let $\omega \in Y\setminus \bigcup_{v\in \mko} Y_v$. Then  $e_v := v(\omega) < \infty$
for each $v\in F$, so   $\bid := \prod_{v\in F} \pid_v^{e_v}$ is an integral ideal.
Let $\aid_{c,1}$ be an ideal of minimal norm in the narrow class of $\bid$, so there exists a totally positive element $x$ such that  $(x) \aid_{c,1} = \bid$.
By assumption $\aid_{c,1}$ is supported in $F$, and so is $\bid$, hence $(x)$ is supported in $F$.
For any $\omega_0\in \Omega_{\aid_{c,1}} \subset Y_0$ we have  $x\omega_0 \in \Omega_{\bid} \subset Y$, so $(x) \in S$.
We also have $v( x\inv \omega) = v(\omega) \geq 0$ for $v\in \mko \setminus F$, because the support of $(x)$ is contained in $F$, and $v( x\inv \omega) = v(\aid_{c,1}) - v(\bid) + v(\omega) = v(\aid_{c,1}) $ for $v\in F$.
This shows that $x\inv \omega \in U_F$, finishing the proof of property (i) in \thmref{KMSandgroundforY}.

By \thmref{KMSandgroundforY} we conclude that the KMS$_\beta$-states of $\heck$ for $\beta\in(1,+\infty]$
are para\-metrized by the tracial states on the groupoid $C^*$-algebra of $\kps/\ops\boxtimes Y_0$. By Corollary~\ref{cor:boundaryalg} the simplex of such tracial states is affinely homeomorphic to the simplex of probability measures on $\clkp\times(\ohs/\opsbar)$.
\end{proof}

\begin{corollary}
The extremal KMS$_\beta$-states of $(\heck,\sigma)$ for $\beta\in(1,+\infty]$ are parametrized by $\clkp\times(\ohs/\opsbar)$.
\end{corollary}

We point out that by class field theory the set $\clkp \times
(\ohs/\opsbar)$ supports a free transitive action of the Galois
group $\gal(\kab/\kk)$ of the maximal abelian extension $\kab$ of
$K$, but, unlike the situation for Bost-Connes type systems, this
does not seem to arise from an action of $\gal(\kab/\kk)$ as
symmetries of $\heck$.

\section{Arithmetic subalgebra and fabulous states}\label{arithmetic}
In addition to the time evolution $\sigma$, the $C^*$-algebra $\heck$
carries several other natural actions.

First, using the presentation of $\heck$ as a semigroup crossed
product $$\opx/\ops\; {}_\alpha\!\!\ltimes C^*(e_r: r\in \kk),$$ we
have the dual action $\hat\alpha$ of the Pontryagin dual of $K^*_+$.
At the level of generators it is given by
\begin{equation*}\label{edual}
\hat\alpha(\eta)\left(\left[\po\begin{pmatrix}1 & y\\ 0 &
x\end{pmatrix}\po\right]\right)=\eta(x)\left[\po\begin{pmatrix}1 & y\\
0 & x\end{pmatrix}\po\right]
\end{equation*}
for characters $\eta$ of $K^*_+$. In order to not confuse Pontryagin
duals with adic completions, we will think of $\hat\alpha$ as a
coaction of $K^*_+$ and call it the \emph{dual coaction}.

Second, the group $\ohs$ acts by the automorphisms $\tau(u)$ defined
by
\begin{equation}\label{esymm}
\tau(u)\left(\left[\po\begin{pmatrix}1 & y\\ 0 &
x\end{pmatrix}\po\right]\right)=\left[\po\begin{pmatrix}1 & u^{-1}y\\
0 & x\end{pmatrix}\po\right],
\end{equation}
where the product $u\inv y$ is defined by approximating $u\inv$ in $\hatok$ by elements
$v$ in $\OO$ and using these elements in the above formula instead
of $u^{-1}$, the result being independent of $v$ sufficiently close
to~$u^{-1}$. Since $\ohs$ is abelian, we could of course have
used~$u$ instead of $u^{-1}$ in the above definition, but our choice
is preferable for consistency with the full Bost-Connes system, as
we will see later.

The action $\tau$ factors through the group
$\ohs/\overline{\OO_+^*}$, which has a class field theoretic
interpretation. Namely, we have a homomorphism $\ohs\to\gal(\kab/K)$
obtained by restricting the Artin map
$r_K\colon\ak^*\to\gal(\kab/K)$ to $\ohs$. Using basic properties of
the Artin map it is easy to check that the kernel of this restriction
 is $\overline{\OO_+^*}$ and that its image is
$\gal(\kab/H_+(K))$, where $H_+(K)$ is the maximal abelian extension
of $K$ that is unramified at all finite places, see,
e.g.,~\cite[Proposition~1.1]{LNT}.

Motivated by~\cite{con-mar} we now give the following definition.

\begin{definition}
Let $\AAA$ be an $\ohs$-invariant $K$-subalgebra of $C_r^*(\Pk, \Po)$. We say that a state $\phi$ on $C_r^*(\Pk, \Po)$ is \emph{fabulous} with respect to $\AAA$ if
$\kab$ is generated by $\phi(\AAA)$ as a $K$-algebra and
$$
\phi(\tau(u)(a))=r_K(u^{-1})(\phi(a))
$$
for all $u\in\ohs$ and $a\in\AAA$. The algebra $\AAA$ is called an \emph{arithmetic subalgebra} if there is a fabulous state with respect to $\AAA$ and the $\C$-algebra generated by $\AAA$ in $C_r^*(\Pk, \Po)$ is dense.
\end{definition}

\begin{remark}
Since $r_K(\ohs) =\gal(\kab/H_+(K))$, it may be more natural to require
that $\AAA$ is an $H_+(K)$-algebra. This is anyway the case for the
algebra $\AAA$ we define below.
\end{remark}

For the full Bost-Connes system  $(A_K, \sigma^K)$, whose construction we recall below, it is shown in~\cite{Y} that there is a canonical choice of an arithmetic subalgebra $\AAA_K$. By~\cite[Theorem~2.4]{LNT} we know also that $C_r^*(\Pk, \Po)$ can be realized as a full corner $p_KA_Kp_K$ of $A_K$. Similarly to the isomorphism in Proposition~\ref{prop:iso}, this realization, however, depends on the choice of a character $\chi$ of $\akf$ to implement the duality between $\kmo$ and $\hatok$.
Our main result about arithmetic subalgebras and fabulous ground states for the Hecke $C^*$-algebra of the orientation preserving affine group of a number field is as follows.
\begin{theorem} \label{thm:arithm}
Let $\AAA_K$ be the arithmetic subalgebra of $A_K$ associated in \cite{Y} to the Bost-Connes system of a number field $\kk$; then the $K$-subalgebra $\AAA\subset C_r^*(\Pk, \Po)$ corresponding to $p_K\AAA_Kp_K$
in the isomorphism of $C_r^*(\Pk, \Po)$ to $p_KA_Kp_K$ has the following properties:
\begin{enumerate}
\item
$\AAA$ is independent of the choice of a character~$\chi$ used to identify $C_r^*(\Pk, \Po)$ with the corner $p_KA_Kp_K$ of the full Bost-Connes algebra. Explicitly, $\AAA$ coincides with the fixed point algebra
$$
\hh_{\kab}(\pk, \po)^\ohs
$$
with respect to an action $\beta$ of the group $\ohs$ on the Hecke algebra $\hh_{\kab}(\pk, \po)$ of the pair $(\pk, \po)$ over the field $\kab$. This action is defined by
$$
\beta(u)\left(a \left[{\Po\left(\begin{array}{ll}1&r\\0&1\end{array}\right)\Po}\right]\right)
=r_K(u)(a)\left[{\Po\left(\begin{array}{ll}1&N_{K\vert \Q}(u)u^{-1}r\\0&1\end{array}\right)\Po}\right]\ \ \text{for}
\ \ a\in\kab,
$$
where $N_{K\vert \Q}$ is the norm map for the field extension $K/\Q$, extended to a map $\aks\to{\mathbb A}_{\Q}^*$, and $r_K(u)(a)$ is simply the Galois action of $r_K(u)$ on $a$.
\item
$\AAA$ is an arithmetic subalgebra of $C_r^*(\Pk, \Po)$. Specifically, every extremal $\sigma$-ground state invariant under the dual coaction of $K^*_+$ is a fabulous state with respect to $\AAA$.
\end{enumerate}
\end{theorem}

Before we embark on the proof of the theorem, we shall revisit the basic facts of the construction of the Bost-Connes system $(A_K, \sigma^K)$ associated to the number field $K$, see \cite{ha-pa,LLN}.


The multiplication action of $\ohs$ on $\akf$ can be induced to an action of $\gal(\kab/K)$ using the homomorphism $\ohs\to\gal(\kab/K)$ obtained from  restricting the Artin map $r_K$ to the subgroup $\ohs$. This gives rise to the balanced product
\[
X_K := \gal(\kab / K)\times_{\ohs} \akf.
\]
The left action of $\ohs$ on $\gal(\kab / K)\times \akf$ that is responsible for the balancing is consistent with  the action of $\jkf$ on $\gal(\kab/K)\times \akf$ given by
\[
g (\gamma, x) = (\gamma
r_K(g)^{-1} , gx)
\]
for $g\in \jkf$ and $x\in \akf$. There is therefore an action of the
quotient group $ \jkf/\ohs$, naturally identified with the discrete
group $J_K$ of fractional ideals in $K$, on $X_K$. Let $J_K^+\subset J_K$ be the subsemigroup of integral ideals. Finally, restricting to the clopen subset
\[
Y_K := \gal(\kab /
K)\times_\ohs \hat\OO
\]
of $X_K$ and letting  $p$ denote the projection in $J_K\ltimes C_0(X_K)$ corresponding to $Y_K$, we obtain the $C^*$-algebra of the Bost-Connes system:
\[
A_K:=C^*(Y_K\boxtimes  J_K) = p\bigl(J_K\ltimes C_0(X_K)\bigr)p\cong J_K^+\ltimes C(Y_K),
\]
where the  isomorphism is from Theorems~2.1 and~2.4 of \cite{minautex}.

The dynamics $\sigma^K$ on $A_K$, and more generally on $J_K\ltimes C_0(X_K)$,  is defined using the absolute norm $N_K$ on $J_K$:
$$
\sigma^K_t(f\lambda_g)=N_K(g)^{it}f\lambda_g
$$
for $f\in C_0(X_K)$ and $g\in J_K$.

The Galois group $\gal(\kab/K)$ acts on $X_K$ by
\begin{equation*}
\gamma (\gamma',m)=(\gamma\gamma',m)
\end{equation*}
for $\gamma,\gamma'\in \gal(\kab/K)$ and $ m\in \akf$. This gives an action $\tau_K$ of $\gal(\kab/K)$ on $A_K$ and $J_K\ltimes C_0(X_K)$ such that
\begin{equation}\label{galois-action-YK}
\tau_K(\gamma)(f\lambda_g)=f(\gamma^{-1}\cdot)\lambda_g
\end{equation}
for  $f\in C(X_K)$, $\gamma\in \gal(\kab/K)$ and $g\in J_K$.

\smallskip

For the proof of \thmref{thm:arithm} we will need the notation and a more explicit description of the isomorphism established in \cite[Theorem 2.4]{LNT} together with some of the details of its proof, so we summarize those results next.
\begin{theorem}{\em(cf. \cite[Theorem 2.4]{LNT})}.\label{thm:iso2}
Denote by $p_K$ the projection in $A_K$ associated to the compact open subset
 \[
Z_{H_+(K)}: =\gal(\kab/H_+(K))\times_{\hatok^*} \hatok
\]
of $Y_K$ and recall that the Artin map induces an isomorphism $\hatok^*/\opsbar \cong \gal(\kab/H_+(K))$.
Then the obvious map $\hatok\to\gal(\kab/H_+(K))\times \hatok$ given by $a\mapsto (e,a)$
gives a homeomorphism of $\hatok/\opsbar$ onto $Z_{H_+(K)}$. Using this homeomorphism and
viewing $\kps/\ops$ as a subgroup of $J_K$ we obtain an isomorphism
$$
p_{\hatok/\overline{\OO_+^*}}\bigl(K_+^*/\OO_+^*\ltimes C_0(\akf/\overline{\OO_+^*})\bigr)p_{\hatok/\overline{\OO_+^*}}\cong p_K A_K p_K
=p_K\bigl(J_K\ltimes C_0(X_K)\bigr)p_K.
$$
Composing this with the isomorphism
$$
C_r^*(\Pk, \Po)\cong   p_{\hatok/\overline{\OO_+^*}}\bigl(K_+^*/\OO_+^*\ltimes C_0(\akf/\overline{\OO_+^*})\bigr)p_{\hatok/\overline{\OO_+^*}}
$$
from \proref{prop:iso} gives an isomorphism
\[
F_K:C_r^*(\Pk, \Po)\to p_K A_K p_K.
\]
\end{theorem}

As an immediate corollary we see that $F_K$ intertwines the action $\tau$ of $\ohs$ on $C_r^*(\Pk, \Po)$ defined in \eqref{esymm} with the action $(\tau_K\circ r_K)|_{\ohs}$ on $p_K A_K p_K$ obtained from \eqref{galois-action-YK}. On the algebra
$$
p_{\hatok/\overline{\OO_+^*}}\bigl(K_+^*/\OO_+^* \ltimes C_0(\akf/\overline{\OO_+^*})\bigr)p_{\hatok/\overline{\OO_+^*}}
$$
the same action is implemented by multiplication by the elements of $\ohs$ on $\akf/\overline{\OO_+^*}$.

\medskip

As for the arithmetic subalgebra of $A_K$, the $K$-subalgebra $\AAA_K$ of $A_K$ is generated by the elements $p\lambda_g p$ for $g\in J_K$ and the algebra of locally constant $\gal(\kab/K)$-equivariant functions $Y_K\to \kab$, see~\cite[Section~10]{Y}. In other words, if we denote by $J_K^+\ltimes C(Y_K,\kab)$ the $\kab$-subalgebra of $A_K=J_K^+\ltimes C(Y_K)$ generated by the elements $p\lambda_g p$ for $g\in J_K$ and the algebra $C(Y_K,\kab)$ of locally constant $\kab$-valued function on $Y_K$, then $\AAA_K$ is the fixed point $K$-subalgebra of $J_K^+\ltimes C(Y_K,\kab)$ under the action $\beta_K$ of~$\gal(\kab/K)$ defined by
$$
\beta_K(\gamma)(fp\lambda_gp)=\gamma(f(\gamma^{-1}\cdot))p\lambda_gp,
$$
for $\gamma \in \gal(\kab/K)$, $g\in J_K$ and $f \in C(Y_K,\kab)$.

In order to understand what happens to the action $\beta_K$ under the isomorphism $F_K$, recall an important observation from \cite{weil}. Denoting by $\rho:\gal(\kab/K)\to \gal(\Q^{\operatorname{ab}}/\Q)$ the homomorphism given by restriction, we have
\begin{equation}\label{eq:weil}
\rho\circ r_K=r_{\Q}\circ N_{K\vert \Q},
\end{equation}
see \cite[Corollary 1, page 246]{weil}. From this we get the following result.

\begin{lemma}[{cf.~\cite[Theorem 4.4]{LvF}}]\label{fromweil}
Under the isomorphism $F_K$ the action $\beta_K|_{\gal(\kab/H_+(K))}$ on $p_K(J_K^+\ltimes C(Y_K,\kab))p_K$ is transformed into the action $\beta$ defined in Theorem~\ref{thm:arithm}.
\end{lemma}

\bp Since for each $y \in \kk$ the image of
$$
R(y)e_y=\left[\po\matr{y}{1} \po\right]
$$
in $C(\hatok/\opsbar)$ under the isomorphism from Proposition~\ref{prop:iso} is an average of finitely many characters of~$\hatok$ of the form $\chi(\cdot vy)$, $v\in\opsbar$, it suffices to show that, for every $u\in\ohs$, we have
$r_K(u)(\chi(u^{-1}z))=\chi(N_{K\vert \Q}(u)u^{-1}z)$ for all $z\in\akf$. Equivalently, we must show that
$$
r_K(u)(\chi(z))=\chi(N_{K\vert \Q}(u)z)\ \ \text{for all}\ \ z\in\akf.
$$
But this follows from \eqref{eq:weil} and the fact that $r_\Q(w)(\chi_\Q(1))=\chi_\Q(w)$ for any character $\chi_\Q$ of $\mathbb{A}_{\Q,f}$ and $w\in\hat\Z^*$.
\ep

\bp[Proof of Theorem~\ref{thm:arithm}]
From the explicit form of the isomorphism $F_K$ given in Theorem~\ref{thm:iso2} and Proposition~\ref{prop:iso} we see that the preimage of $p_K(J_K^+\ltimes C(Y_K,\kab))p_K$ in $\heck$ coincides with the Hecke algebra $\hh_{\kab}(\pk, \po)$, since the functions $\chi(\cdot y)$ for $y\in\OO$ span over $\Q^{\operatorname{ab}}$ the space of locally constant $\Q^{\operatorname{ab}}$-valued functions on $\hatok$, hence they span over $\kab$ the space of locally constant $\kab$-valued functions on $\hatok$. Part (1) of the theorem now follows from \lemref{fromweil}.

In order to prove the second part of Theorem~\ref{thm:arithm} it
remains to show that every extremal $\sigma$-ground state invariant
under the dual coaction of $K^*_+$ is a fabulous state. By
Theorem~\ref{thm:gs-for-redHecke} the ground states correspond to
the states of $C^*(\kps/\ops \boxtimes Y_0)$, among which the ones that are invariant
under the dual coaction of $K^*_+$ correspond to the probability
measures on $Y_0$. The ground state corresponding to such a measure
is extremal if and only if the measure is a point mass.
So the space of extremal ground states that are invariant under the dual coaction
can be identified with $Y_0$.

Using our identification of $\hat\OO/\overline{\OO_+^*}$ with
$Z_{H_+(K)}=\gal(\kab/H_+(K))\times_{\hatok^*} \hatok$ we view $Y_0$
as a subset of $Z_{H_+(K)}$. Then in order to finish the proof it
suffices to check the following. If $\B$ is the $K$-algebra of
locally constant $\gal(\kab/H_+(K))$-equivariant functions
$Z_{H_+(K)}\to \kab$, then for every point $z\in Y_0\subset
Z_{H_+(K)}$ the set $\{f(z)\mid f\in\B\}$ generates $\kab$ as a
$K$-algebra. The proof of this is similar to the analogous property of
$\AAA_K$ in~\cite{Y} and  relies on the following two key
points: the $\gal(\kab/H_+(K))$-space $Z_{H_+(K)}$ is a projective
limit of finite $\gal(\kab/H_+(K))$-spaces $Z_n$ and every point
$z\in Y_0$ has trivial stabilizer in $\gal(\kab/H_+(K))$. To see this,
take any $a\in\kab$ and let~$H$ be the stabilizer of $a$ in
$\gal(\kab/H_+(K))$. Since $H$ is an open subgroup of the compact
group $\gal(\kab/H_+(K))$, there is a sufficiently large $n$ such that
the stabilizer of the image $z_n$ of $z$ in $Z_n$ is contained in
$H$. We can then define a $\gal(\kab/H_+(K))$-equivariant function
$f_n$ on $Z_n$ such that $f_n(z_n)=a$ by letting $f_n(\gamma
z_n)=\gamma(a)$ and $f=0$ outside the orbit of $z_n$. Then composing
$f_n$ with the projection $Z_{H_+(K)}\to Z_n$ we get a function
$f\in\B$ such that $f(z)=a$. \ep

\end{document}